  \documentclass[a4paper,10pt,reqno]{amsart}

\usepackage{amsmath}
\usepackage{amsthm}
\usepackage{amssymb}
\usepackage{enumitem}
\usepackage[all]{xy}
\usepackage{hyperref}

\numberwithin{equation}{section}
\setlist{leftmargin=*,label={\rm(\arabic*)}}

\theoremstyle{plain}

\newtheorem{thm}{Theorem}[section]
\newtheorem{prop}[thm]{Proposition}
\newtheorem{cor}[thm]{Corollary}
\newtheorem{lem}[thm]{Lemma}
\theoremstyle{definition}

\newtheorem{exm}[thm]{Example}
\theoremstyle{remark}
\newtheorem{rmk}[thm]{Remark}

\DeclareMathOperator{\Hom}{Hom}
\DeclareMathOperator{\Rad}{Rad}
\DeclareMathOperator{\Ext}{Ext}

\DeclareMathOperator{\uRad}{\underline{Rad}}

\DeclareMathOperator{\Ker}{Ker}
\DeclareMathOperator{\Coker}{Coker}

\renewcommand{\Im}{\operatorname{Im}}

\DeclareMathOperator{\rad}{rad}

\DeclareMathOperator{\Diff}{Diff}
\DeclareMathOperator{\uDiff}{\underline{Diff}}
\def\presuper#1#2%
  {\mathop{}%
   \mathopen{\vphantom{#2}}^{#1}%
   \kern-\scriptspace%
   #2}
\def\presub#1#2%
  {\mathop{}%
   \mathopen{\vphantom{#2}}_{#1}%
   \kern-\scriptspace%
   #2}

\newcommand{\lto}{\longrightarrow}
\mathchardef\mhyphen="2D
\renewcommand{\-}{\mhyphen}

\newcommand{\op}{\mathrm{op}}

\newcommand{\Mod}{\mathrm{Mod}}
\newcommand{\proj}{\mathrm{proj}}
\newcommand{\Proj}{\mathrm{Proj}}

\newcommand{\GProj}{\mathrm{GProj}}

\newcommand{\uGProj}{\underline{\mathrm{GProj}}}

\newcommand{\bbZ}{\mathbb{Z}}

\renewcommand{\S}{\mathcal{S}}
\newcommand{\T}{\mathcal{T}}

\begin{document}
\title[Differential projective modules]
      {Differential projective modules over algebras with radical square zero}
\author[Dawei Shen]{Dawei Shen}
\address{School of Mathematics and Statistics \\
         Henan University \\
         Kaifeng, Henan 475004 \\
         P. R. China}
\email{dwshen@vip.henu.edu.cn}
\subjclass[2010]{Primary 16G10; Secondary 16G50, 18G25}
\keywords{Gorenstein projective module, representation of quivers, differential module,
algebra with radical square zero, virtually Gorenstein algebra}
\date{\today}

\begin{abstract}
  Let $Q$ be a finite quiver and $\Lambda$ be the radical square zero algebra of $Q$ over a field.
  We give a full and dense functor from the category of reduced
  differential projective modules over $\Lambda$
  to the category of representations of the opposite of $Q$.
  If moreover $Q$ has oriented cycles and $Q$ is not a basic cycle,
  we prove that the algebra of dual numbers over $\Lambda$ is not virtually Gorenstein.
\end{abstract}
\maketitle

\section{Introduction}
Given a ring $\Lambda$,
recall that a {\em differential $\Lambda$-module} is a pair $(M, d)$,
where $M$ is a $\Lambda$-module
and $d$ is an endomorphism of $M$ such that $d^2 = 0$.
If $(M, d),(N, \delta)$ are differential $\Lambda$-modules,
a {\em differential $\Lambda$-module map} $f\colon (M, d)\to (N, \delta)$
is  a $\Lambda$-module map $f\colon M \to N$ such that $f d = \delta f$.
A  $\Lambda$-module map $f\colon (M, d)\to (N, \delta)$ is {\em null homotopic}
if there exists a $\Lambda$-module map
$r\colon M \to N$ such that $f = r d + \delta r$.
Note that  differential modules over $\Lambda$ are just   modules
over  the ring  $\Lambda[\epsilon]$ of dual numbers over $\Lambda$.

Let $(M,d)$ be a differential $\Lambda$-module.
The {\em shift}  $\Sigma(M,d)$  of $(M,d)$ is  $(M, -d)$.
Recall that   $(M,d)$ is {\em contractible}
if the identity map of $M$ is null homotopic,
$(M,d)$ is  {\em reduced} if it has no nonzero contractible direct summands,
and $(M,d)$ is {\em exact} if $\Ker d = \Im d$.
Every contractible differential $\Lambda$-module is exact.

By the term {\em differential projective $\Lambda$-modules},
we mean differential $\Lambda$-modules
which are projective as $\Lambda$-modules.
In ~\cite{RZ2017}  differential finitely generated projective $\Lambda$-module
are called perfect differential $\Lambda$-modules.

The graded differential modules, namely the complexes, have been studied by many authors.
However, few  papers investigate the differential modules in detail.
L. ~L.  ~Avramov, R.-O. ~Buchweitz and S. ~Iyengar ~\cite{ABI2007}
study the projective class as well as  free class and flat class for differential modules.
C. ~M.  ~Ringel and P. ~Zhang ~\cite{RZ2017} investigate
the perfect differential modules over  path algebras,
they prove that the homology functor gives a bijection
from the reduced perfect differential modules
to the finitely generated modules  over  path algebras.
J. Wei ~\cite{Wei2015} studies the Gorenstein homological theory
for differential modules and extends their bijection.

The study of differential modules is  related to
the Gorenstein homological theory.
M. ~Auslander and M. ~Bridger ~\cite{AB1969} introduce
the notion of modules of G-dimension zero over two-sided Noetherian rings.
This kind of modules are also called totally reflexive modules ~\cite{AM2002}.
E. ~E. ~Enochs and O. ~M. ~G. ~Jenda  ~\cite{EJ1995,EJ2000} extend their ideas and
introduce the notion  of Gorenstein projective modules,
Gorenstein injective modules and Gorenstein flat modules for arbitrary rings.
In particular, the
totally reflexive modules are just
the finitely generated  Gorenstein projective
modules for two-sided Noetherian rings.

The Gorenstein projective modules over algebras
with radical square zero have been well studied.
X.-W. ~Chen ~\cite{Che2012} shows that
a connected  Artin algebra  with radical square zero is either selfinjective or CM-free.
Recall that an Artin algebra is called CM-free if every totally reflexive module is projective.
C. ~M. ~Ringel and B.-L. ~Xiong ~\cite{RX2012} extend this result to arbitrary Gorenstein projective modules.

However, the Gorenstein projective modules over  algebras with radical cubic zero are quite complicated.
Y. ~Yoshino ~\cite{Yos2003} studies a class of commutative local Artin algebras with radical cubic zero,
over these algebras the  simple  module has no right approximations by the totally reflexive modules.
They are  not virtually Gorenstein algebras in the sense of ~\cite{Bel2005}.

The totally reflexive modules over $S_n=k[X,Y_1,\cdots,Y_n]/(X^2,Y_iY_j)$ are studied by
D. ~A. ~Rangel  ~Tracy ~\cite{Ran2015},
where  $k$ is a field, $n\geq 2$ and $1\leq i,j\leq n$.
The main result gives a bijection  from the reduced totally reflexive modules over $S_n$
to the  finite-dimensional modules over the free algebra  of $n$ variables.

Inspired their works,
we investigate the differential projective modules
over Artin algebras with radical square zero.

Let $k$ be a field and $Q$ be a finite quiver.
Denote by $kQ$ the path algebra of $Q$.
Let $J$ be the arrow ideal of $kQ$,
then the quotient algebra $kQ/J^2$ is an Artin algebra with radical square  zero.
Let $Q^\op$ be the opposite quiver of $Q$.

We construct a ``Koszul dual functor"  $F$
from the category of reduced differential projective $kQ/J^2$-modules
to the category of $kQ^\op$-modules.

Denote by $\Diff(kQ/J^2\-\Proj)$ the Frobenius category
of all differential projective $kQ/J^2$-modules \cite{Hap1988}.
The homotopy category $\uDiff(kQ/J^2\-\Proj)$
of all differential projective $kQ/J^2$-modules
 is a triangulated category \cite{Ver1996}.

Denote by $\Diff_0(kQ/J^2\-\Proj)$ the full subcategory
of $\Diff(kQ/J^2\-\Proj)$ formed by  reduced differential projective $kQ/J^2$-modules.
We recall  the abelian category $kQ^\op\-\Mod$  of all $kQ^\op$-modules.

The following is the main result of this paper; see also \cite{Ran2015}, compare \cite{BGS1996,RZ2017}.

\begin{thm} \label{thm:mainthm1}
Let $k$ be a field and $Q$ be a finite quiver.
Then taking the top makes a functor
$F\colon \Diff_0(kQ/J^2\-\Proj) \to kQ^\op\-\Mod$ from the category of reduced differential projective $kQ/J^2$-modules
to the category of $kQ^\op$-modules.
Moreover,
\begin{enumerate}
  \item $F$ is full, dense, and detects the isomorphisms;
  \item $F$ is exact and commutes with all small coproducts;
  \item $F$ vanishes on all null-homotopic maps;
  \item For any $M, N$ in $\Diff_0(kQ/J^2\-\Proj)$, there is an isomorphism
  \[{\uDiff(kQ/J^2\-\Proj)}(M, N) \simeq
  \Hom_{kQ^\op}(F(M), F(N)) \coprod \Ext^1_{kQ^\op}(F(M), F\Sigma (N)).\]
\end{enumerate}
\end{thm}

The ``Koszul dual functor" $F$ has a good restriction on some full subcategories.
More precisely, we have the following.

\begin{prop} \label{prop:mainprop2}
Let $M$ be a reduced differential projective $kQ/J^2$-module in the homotopy category $\uDiff(kQ/J^2\-\Proj)$.
\begin{enumerate}
  \item $M$ is finite dimensional if and only if $F(M)$ is  finite dimensional.
  \item $M$ is compact  if and only if $F(M)$ is finitely presented.
  \item $M$ is exact if and only if $\Ext^n_{kQ^\op}(kQ_0, F(M))=0$ for $n=0,1$.
\end{enumerate}
\end{prop}

We give a compact generator for the homotopy category $\uDiff(kQ/J^2\-\Proj)$ as follows.
Let $C$ be the $kQ/J^2$-module $kQ/J^2 \otimes_{kQ_0} kQ^\op$
with a differential $d$ given by
$d (y\otimes z)=\sum_{\alpha\in Q_1} y\alpha\otimes \alpha^*z$
for $y\in  kQ/J^2,z\in kQ^\op$.
Here, $\alpha^* \in Q^\op$ is the reversed arrow  of $\alpha$.

We have the following.

\begin{thm} \label{thm:mainthm4}
The above $C$ is a compact generator for $\uDiff(kQ/J^2\-\Proj)$.
\end{thm}

Recall that a finite connected quiver $Q$ is a {\em basic cycle} if the number of vertices
is equal to the number of arrows in $Q$
and all arrows  form an oriented cycle.

The following gives a class of noncommutative algebras with  radical cubic zero
which are not  virtually Gorenstein algebras; compare ~\cite{Yos2003}.

\begin{thm}  \label{thm:mainthm3}
If $Q$ is a finite connected quiver with oriented cycles and $Q$ is not a basic cycle.
Then the algebra $kQ/J^2[\epsilon]$ of dual number over $kQ/J^2$ is not virtually Gorenstein.
\end{thm}

The present paper is organized as follows.
In Section ~2 and Section ~3, we recall some required facts about quivers and
radical square zero algebras, respectively.
In Section ~4, we construct the functor $F$ and prove  Theorem ~1.1.
In Section ~5, we study the restriction of $F$ and give the proofs of  Proposition ~1.2 and Theorem ~1.3.
In Section ~6, we study virtual Gorensteinness of algebras and prove Theorem ~1.4.

\section{Quivers and representations} \label{sec:pre}
In this section,
we recall some  facts about quivers and their representations.
We refer to \cite[III.1]{ARS1995} for more details.

A {\em finite quiver} $Q$ is a quadruple $(Q_0,Q_1;s,t)$,
where $Q_0$ is the finite set of vertices, $Q_1$ is the finite set of arrows,
and $s, t \colon Q_1 \to Q_0$ are the source  map and the target  map, respectively.
Denote by $e_i$ the {\em trivial path} at $i$ for $i\in Q_0$, where $s(e_i)=t(e_i)=i$.
A {\em nontrivial path} $p$ is a sequence  $\alpha_l\cdots \alpha_2\alpha_1$ of arrows,
where $l\geq 1$ and $t(\alpha_i)=s(\alpha_{i+1})$ for $1\leq i\leq l-1$.
Here, $s(p)=s(\alpha_1)$ and $t(p)=t(\alpha_l)$.
A nontrivial path $p$ is called an {\em oriented cycle} if $s(p)=t(p)$.
A finite quiver $Q$ is said to be {\em acyclic} if it has no oriented cycles.

Let $kQ$ be the  path algebra  of $Q$ over a field $k$.
Recall that  $kQ$ is a hereditary algebra
and $kQ$ is finite dimensional if and only if $Q$ is acyclic.

A {\em representation} of $Q$
is a collection $(M_i, M_\alpha)_{i\in Q_0, \alpha\in Q_1}$,
where $M_i$ is a $k$-vector space and
$M_\alpha$ is a $k$-linear map from $M_{s(\alpha)}$ to $M_{t(\alpha)}$.
If $M,N$ are representations of $Q$,
a morphism from $M$ to $N$ is  a collection
$(f_i)_{i\in Q_0}$, where $f_i$ is a $k$-linear map from $M_i$ to $N_i$ such that
$ f_{t(\alpha)}M_\alpha = N_\alpha f_{s(\alpha)}$ for each $\alpha\in Q_1$.

Recall that the category of representations of $Q$  is equivalent to
the category of  $kQ$-modules.
We identify a $kQ$-module with the associated  representation of $Q$.

Let $J$ be the {\em arrow ideal} of $kQ$; it is the ideal generated by all arrows in $Q$.
For  an {\em admissible} ideal $I$ of $kQ$ satisfying $J^n \subseteq I \subseteq J^2$ for some $n\geq 2$,
the quotient algebra $kQ/I$ is   finite dimensional  over $k$.
Recall that a $kQ/I$-module is  a $kQ$-module $M$ such that $wM= 0$
for every $w\in I$.

\section{Projective modules for radical square zero algebras} \label{sec:algsquare}
Let $k$ be a field and $Q$ be a finite quiver.
Recall that the unit element of the path algebra $kQ$ is $\sum_{i\in Q_0}e_i$,
where $e_i$ the trivial path at vertex $i$.

Let $J$ be arrow ideal of $kQ$, then $kQ/J^2$ is the radical square zero algebra of $Q$.
In this section, we study some facts about the projective $kQ/J^2$-modules.

\begin{lem} \label{lem:projpro}
Let $M$ be a  projective  $kQ/J^2$-module. Then we have
\begin{enumerate}
  \item $\rad^2(M) = 0$;
  \item $(\rad M)_i = \coprod_{\alpha\in Q_1,t(\alpha)=i} \Im M_\alpha$ for each $i\in Q_0$;
  \item $(\rad M)_{s(\alpha)} =\Ker M_\alpha$ for each $\alpha\in Q_1$;
  \item $(M/\rad M)_{s(\alpha)} = M_{s(\alpha)}/\Ker M_\alpha$  for each $\alpha\in Q_1$.
\end{enumerate}
\end{lem}

\begin{proof}
Observe that   (1)--(4) hold for the regular module over $kQ/J^2$.
Since $M$ is projective over $kQ/J^2$,
it is  a direct summand of direct sums of copies of $kQ/J^2$.
Since taking the radicals,  kernels and images commute with all small coproducts,
we infer that  (1)--(4) also hold  for $M$.
\end{proof}

Given a module $M$,
recall that the {\em radical}  $\rad M$ of  $M$
is the intersection of all maximal submodules of $M$,
and the {\em top} of $M$ is the  quotient module $M/\rad M$.

Let $f\colon M\to N$ be a $kQ/J^2$-module map between  projective $kQ/J^2$-modules.
Recall that $f$ is  {\em radical} if $\Im f \subseteq \rad N$.
Denote by $F(f)\colon M/\rad M \to N/ \rad N$ the induced map of $f$.
It follows that $f$ is radical if and only if $F(f)=0$.
Denote by $\Rad(M, N)$ the subspace of
$\Hom_{kQ/J^2}(M, N)$ formed by radical  maps.

\begin{lem} \label{lem:full}
For projective $kQ/J^2$-modules $M,N$, there is a short exact sequence
\[0 \to \Rad(M, N) \overset{\subset}\lto \Hom_{kQ/J^2}(M, N) \overset{F}\lto \Hom_{kQ_0}(M/\rad M, N/\rad N) \to 0.\]
\end{lem}

\begin{proof}
Let $g\colon  M/\rad M \to N/\rad N$  be a $kQ_0$-module map.
Since $M$ is projective over $kQ/J^2$,
we have $g = F(f)$ for some $kQ/J^2$-module map $f \colon M \to N$.
Then the map $F$ is surjective.
Since  $\Ker F = \Rad(M, N)$,
this gives rise to the desired short exact sequence.
\end{proof}

For $kQ_0$-modules $X$ and $Y$,
denote by $\mathrm{E}(X, Y)$  the $k$-vector space consisting of the collections $(f_{\alpha^*})_{\alpha \in Q_1}$,
where $f_{\alpha^*}\colon X_{t(\alpha)}\to Y_{s(\alpha)}$
is a $k$-linear map.

\begin{lem} \label{lem:gamma}
For projective $kQ/J^2$-modules $M,N$, there is an isomorphism
\[\gamma\colon\Rad(M, N) \overset{F}\lto
\Hom_{kQ_0}(M/\rad M, \rad N) \overset{G}\lto \mathrm{E}(M/\rad M, N/\rad N).\]
\end{lem}

\begin{proof}
We identify $\Rad(M, N)$ with $\Hom_{kQ/J^2}(M, \rad N)$.
Then the map $F$ is well defined since $\rad^2(N)$ is zero.
By Lemma ~\ref{lem:full} the map $F$ is  surjective and $\Ker F = 0$.
Then $F$ is an isomorphism.

We recall Lemma ~\ref{lem:projpro}(1)--(4).
Denote by $p_{\alpha}\colon (\rad N)_{t(\alpha)} \to \Im N_{\alpha}$  the  natural projection
and by $i_{\alpha}\colon\Im N_{\alpha} \to (\rad N)_{t(\alpha)}$ the natural inclusion.
Let us denote by $\overline{N_\alpha}\colon N_{s(\alpha)}/\Ker N_\alpha \to \Im N_{\alpha}$
the induced isomorphism of  $N_\alpha$.

Let $g \in \Hom_{kQ_0}(M/\rad M, \rad N)$.
Define $G(g)_{\alpha^*} = (\overline{N_\alpha})^{-1}p_\alpha g_{t(\alpha)}$ for $\alpha \in Q_1$.
Let $h\in \mathrm{E}(M/\rad M, N/\rad N)$.
Define ${G^{-1}}(h)_{i} = \sum_{\alpha} i_{\alpha} \overline{N_{\alpha}}h_{\alpha^*}$ for $i\in Q_0$,
where $\alpha$ runs through all arrows terminating at $i$.
One checks that $G$ and $G^{-1}$ are mutually inverse  isomorphisms.

Therefore, the composite $\gamma =G \circ F$ is an isomorphism.
\end{proof}

Recall the {\em opposite} quiver $Q^\op$ of $Q$.
The underlying graph of $Q^\op$ is the same as $Q$,
but the orientations are all reversed.
We denote by $\alpha^*$ the reversed arrow in $Q^\op$ for each arrow $\alpha$ in $Q$.

If $X, Y$ are $kQ^\op$-modules,
let $\mathrm{E}_0(X, Y)$ be the subspace of $\mathrm{E}(X, Y)$ formed by $(h_{\alpha^*})_{\alpha \in Q_1}$,
where $h_{\alpha^*}= \theta_{s(\alpha)}X_{\alpha^*}- Y_{\alpha^*}\theta_{t(\alpha)}$
for some $kQ_0$-module map $\theta\colon X \to Y$.

We need the following lemmas.

\begin{lem} \label{lem:extension}
For $kQ^\op$-modules $X,Y$,
there is an isomorphism
\[\Ext^1_{kQ^\op}(X, Y) \simeq \mathrm{E}(X, Y)/\mathrm{E}_0(X, Y).\]
\end{lem}

\begin{proof}
This follows from ~\cite[7.2]{GR1992}.
\end{proof}

\begin{lem}  \label{lem:comhe}
Let $X$ be a $kQ^\op$-module.
Then $X$ is finitely presented if and only if
the functors $\Hom_{kQ^\op}(X,-)$ and $\Ext^1_{kQ^\op}(X,-)$ commute with all small coproducts.
\end{lem}

\begin{proof}
Since the path algebra $kQ^\op$ is hereditary,
the projective dimension of $X$ is no more than one.
Then the statement follows from  ~\cite[1.4 Corollary 2]{Str1976}.
\end{proof}

\section{Construction of the ``Koszul dual functor"}
In this section, for a finite quiver $Q$
we show that taking the top  makes a full and dense functor
from the category of reduced differential projective $kQ/J^2$-modules
to the category of $kQ^\op$-modules.
Here, $Q^\op$ is the opposite of $Q$.

Let $M$ and $N$ be differential projective $kQ/J^2$-modules.
A $kQ/J^2$-module map $f\colon M\to N$ is said to be null homotopic
if there is a $kQ/J^2$-module map $r\colon M \to N$
such that $f =  r d + \delta r$.
Here, $d$ and $\delta$ are the differentials of $M$ and $N$, respectively.
Denote by $\mathrm{Hpt}(M, N)$ the subspace of $\Hom_{kQ/J^2}(M, N)$ formed by null-homotopic maps.

Let $M$ be a differential projective $kQ/J^2$-module.
Recall that $M$ is said to be contractible if the identity map of $M$ is null homotopic,
and  $M$ is said to be reduced if $M$ has no nonzero  contractible direct summands.

We need the following.

\begin{lem} \label{lem:min}
Let $M$ be a differential projective $kQ/J^2$-module.
\begin{enumerate}
  \item $M$ is reduced if and only if the differential  of $M$ is a radical map.
  \item There exists a decomposition $M = M' \coprod M''$  such that $M'$ is contractible and $M''$ is reduced.
  Moreover, this decomposition is unique up to isomorphism.
\end{enumerate}
\end{lem}

\begin{proof}
View $M$ as a one-periodic complex
$\cdots \overset{d} \to  M \overset{d} \to   M \overset{d} \to  \cdots$,
where $d$ is the differential.
Then (1) and (2) follow from the dual versions of ~\cite[Appendix ~B]{Kra2005}.
\end{proof}

Let us recall some notations.
We denote by $\Diff(kQ/J^2\-\Proj)$ the category of all differential projective $kQ/J^2$-modules.
It is a Frobenius category and it admits all small products.
Denote by $\Diff_0(kQ/J^2\-\Proj)$ the full subcategory consisting of reduced differential projective $kQ/J^2$-modules.

Recall the homopopy category $\uDiff(kQ/J^2\-\Proj)$ of all differential projective $kQ/J^2$-modules.
The objects are all differential projective $kQ/J^2$-modules.
The morphisms are obtained from differential $kQ/J^2$-module maps by factoring out
the null homotopic maps.

Note that $\Diff_0(kQ/J^2\-\Proj)$ is not extension closed in $\Diff(kQ/J^2\-\Proj)$.
However, the homotopy categories of these two categories are equivalent.

Let $(M, d)$ and $(N, \delta)$ be reduced differential projective $kQ/J^2$-modules.
Then we have inclusions
\[\mathrm{Hpt}(M, N) \subseteq \Rad(M, N) \subseteq {\Diff(kQ/J^2\-\Proj)}(M, N).\]

In fact, since $M$ and $N$ are reduced,
by Lemma ~\ref{lem:min}(1) $d$ and $\delta$ are radical maps.
If $f\colon M \to N$ is radical, then $fd = 0 = \delta f $ and thus $f$ is a differential map.
Then we obtain the   inclusion on the right hand side.
Similarly, the inclusion on the left hand side also holds.

We now prove  the following  key lemma.
Here, we recall the maps $F$ and  $\gamma$  from Lemma ~\ref{lem:gamma}.

\begin{lem} \label{lem:gammakey}
Let $f\colon M \to N$ be a $kQ/J^2$-module map between reduced differential projective $kQ/J^2$-modules.
Then we have
\begin{enumerate}
  \item $f$ is a differential map if and only if $F(f)\gamma(d) =\gamma(\delta)F(f)$;
  \item $f$ is null homotopic if and only if $F(f) =0$ and there exists a $kQ_0$-module map
        $\theta\colon M/\rad M \to N/\rad N$ such that
        $\gamma(f) = \theta\gamma(d) + \gamma(\delta)\theta$.
\end{enumerate}
\end{lem}

\begin{proof}
(1) Since  the map $\gamma$ is an isomorphism by Lemma ~\ref{lem:gamma},
we infer that $fd =\delta f$ if and only if $\gamma(fd)=\gamma(\delta f)$.
Note that $\gamma(fd)=F(f)\gamma(d)$  and $\gamma(\delta f)=\gamma(\delta)F(f)$.
Then $f$ is a differential map if and only if $F(f)\gamma(d) =\gamma(\delta)F(f)$.

(2) $``\Longrightarrow"$ Since $f$ is null homotopic,
there is a $kQ/J^2$-module map $r\colon M   \to N$ such that $f=rd+\delta r$.
Then $F(f)= 0$ and $\gamma(f) = F(r)\gamma(d) + \delta F(r)$.

$``\Longleftarrow"$
Since $M$ is projective, there is a $kQ/J^2$-module map $r\colon M   \to N$ such that $\theta=F(r)$.
Note that $f$ is radical and $\gamma(f)=\gamma(rd+\delta r)$.
Then $f=rd+\delta r$  by Lemma ~\ref{lem:gamma}.
We infer that $f$ is null homotopic.
\end{proof}

We now construct the ``Koszul dual functor" $F$ from the category
of reduced  differential projective $kQ/J^2$-modules
to  the category of  $kQ^\op$-modules.

For any object $(M, d)$ in $\Diff_0(kQ/J^2\-\Proj)$,
set $F(M,d)_i = (M/\rad M)_i$ for  $i \in Q_0$
and $F(M,d)_{\alpha^*} = \gamma(d)_{\alpha^*}$ for  $\alpha \in Q_1$.
Then $F(M,d)$ is a $kQ^\op$-module.

For a morphism $f$ in $\Diff_0(kQ/J^2\-\Proj)$,
recall that $F(f)$  is a $kQ_0$-module map.
It follows from Lemma ~\ref{lem:gammakey}(1) that $F(f)$ is  a $kQ^\op$-module map.

Let $\sigma$ be the algebra isomorphism
of $kQ^\op$ induced by $\sigma(q) = (-1)^lq$
for every path $q$ in  $Q^\op$, where $l$ is the length of $q$.
Note that $\sigma^2$ is the identity map.

For a $kQ^\op$-module $X$,
let $\presuper{\sigma}X$ be the {\em twisted module} of $X$.
Here, $\presuper{\sigma}X$ is equal to $X$ as $k$-vector spaces,
and the action $\circ$ is given by $w \circ x = \sigma(w)x$ for $w \in kQ^\op, x\in X$.

We see that $\presuper{\sigma}{(\presuper{\sigma} X)}$ is the same as $X$.
However, the twisted module $\presuper{\sigma} X$ and the original module $X$  need not be isomorphic.
The following is an example.

\begin{exm}
Let $k$ be field and $Q$ be the following quiver.
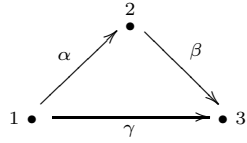
\[ \xymatrix{&   {}_\bullet^2\ar[dr]^\beta  \\
\mbox{\scriptsize{1 $\bullet$} }\ar[rr]_\gamma\ar[ur]^\alpha &  & \mbox{\scriptsize{$\bullet$ 3}}
  } \]

Let $X$ be the $kQ$-module with $X_1=X_2=X_3=k$, $X_{\alpha}=X_{\beta}=X_{\gamma}=1_k$,
where $1_k$ is the identity map.
Then $\presuper\sigma{X}_1= \presuper\sigma{X}_2= \presuper\sigma{X}_3=k$,
$\presuper\sigma{X}_{\alpha}= \presuper\sigma{X}_{\beta}= \presuper\sigma{X}_{\gamma}=-1_k$.

If the characteristic of  $k$ is not equal to $2$,
then the two $kQ$-modules $X$ and $\presuper\sigma X$ are not isomorphic.
\end{exm}

Let $M$ be a reduced differential projective $kQ/J^2$-module.
The shift $\Sigma(M)$ of $M$ is the equal to $M$ as $kQ/J^2$-modules, while
the differential of $\Sigma(M)$ is the negative of the differential of $M$.
Observe that the $kQ^\op$-modules $\presuper{\sigma}F(M)$  and $F\Sigma (M)$ are isomorphic.

We have the following.

\begin{lem} \label{lem:rad-ext}
For any $M, N$  in $\Diff_0(kQ/J^2\-\Proj)$,
there is an isomorphism
\[\uRad(M, N) \simeq \Ext^1_{kQ^\op}(F(M),F\Sigma (N)).\]
Here, we write $\uRad(M, N) = \Rad(M, N)/\mathrm{Htp}(M, N)$.
\end{lem}

\begin{proof}
It follows from Lemma ~\ref{lem:gamma} that there is an isomorphism
\[\gamma\colon \Rad(M, N)\overset{\sim}\lto \mathrm{E}(F(M), F(N))= \mathrm{E}(F(M), F\Sigma (N)).\]
The image of $\mathrm{Hpt}(M,N)$ under $\gamma$ is $\mathrm{E}_0(F(M),F\Sigma (N))$
by Lemma ~\ref{lem:gammakey}(2).
We infer from Lemma ~\ref{lem:extension} that $\uRad(M, N)$ and  $\Ext^1_{kQ^\op}(F(M),F\Sigma (N))$ are isomorphic.
\end{proof}

The following is the main result of this section.

\begin{thm} \label{thm:repiso}
Let $k$ be a field and $Q$ be a finite quiver.
Then taking the top makes a functor
$F\colon \Diff_0(kQ/J^2\-\Proj) \to kQ^\op\-\Mod$ from the category of reduced differential projective $kQ/J^2$-modules
to the category of $kQ^\op$-modules.
Moreover,
\begin{enumerate}
  \item $F$ is full, dense, and detects the isomorphisms;
  \item $F$ is exact and commutes with all small coproducts;
  \item $F$ vanishes on all null-homotopic maps;
  \item For any $M, N$ in $\Diff_0(kQ/J^2\-\Proj)$, there is an isomorphism
  \[{\uDiff(kQ/J^2\-\Proj)}(M, N) \simeq
  \Hom_{kQ^\op}(F(M), F(N)) \coprod \Ext^1_{kQ^\op}(F(M), F\Sigma (N)).\]
\end{enumerate}
\end{thm}

\begin{proof}
$(1)$  Let $M$ and $N$ be in $\Diff_0(kQ/J^2\-\Proj)$
and let $g\colon F(M) \to F(N)$  be a $kQ^\op$-module map.
By Lemma ~\ref{lem:full} there is a $kQ/J^2$-module map
$f\colon M \to N$ such that $g = F(f)$.
Since $g$ is a $kQ^\op$-module map,
it follows from Lemma ~\ref{lem:gammakey}(1) that $f$ is a differential map .
This shows that the functor $F$ is full.

For a $kQ^\op$-module $X$, set $G(X)$ be the $kQ/J^2$-module
$kQ/J^2\otimes_{kQ_0}X$ with a differential $d$ given by
\[  d( \sum_{i\in Q_0}y_i\otimes x_i)
   =\sum_{\alpha \in Q_1}y_{t(\alpha)}\alpha\otimes X_{\alpha^*}(x_{t(\alpha)})\]
for $y_i\in (kQ/J^2)e_i$ and $x_i\in X_i$.
Here, we recall the target $t(\alpha)$ of the arrow $\alpha$.

Note that $G(X)$ is a reduced differential projective $kQ/J^2$-module
and $F(G(X))$ is isomorphic to $X$.
It follows that the functor $F$ is dense.

It remains to show that $F$ detects the isomorphisms.
Suppose that $f\colon M \to N$ is a morphism in
$\Diff_0(kQ/J^2\-\Proj)$  with $F(f)$ being an isomorphism.

Let $g$ be the inverse of $F(f)$.
Since $N$ is projective, there exists a morphism $h\colon N\to M$ such that $F(h)= g$.
Then $F(1_M) = F(hf)$, $F(1_N) = F(fh)$,
where $1_M$ and $1_N$ are the identity maps.
By Lemma ~\ref{lem:full} both
$1_M - hf$ and  $1_N - fh$ are radical maps.
Since $\rad^2(M) = 0$, $\rad^2(N)=0$,
we have $(1_M - hf)^2 = 0$, $(1_N -fh)^2 = 0$.
Then $hf$ and $fh$ are isomorphisms.
It follows that $f$ is an isomorphism.

(2) For every  reduced differential projective $kQ/J^2$-module $M$,  recall that
$F(M)$ is isomorphic to $kQ_0\otimes_{kQ/J^2}M$ as $k$-vector spaces.
It follows that  $F$ is exact and commutes with all small coproducts.

(3) Since every null-homotopic map is radical and $F$ vanishes on all radical maps,
it follows that $F$ vanishes on all null-homotopic maps.

(4)  Let  $M$ and $N$  be  in $\Diff_0(kQ/J^2\-\Proj)$.
Since the functor $F$ is full by (1),
there is a short exact sequence
\[0 \to \Rad(M, N)  \to \Hom_{\Diff(kQ/J^2\-\Proj)}(M, N) \to \Hom_{kQ^\op}(F(M), F(N))\to 0.\]
Since $\mathrm{Hpt}(M, N) \subseteq \Rad(M, N)$, it yields a short exact sequence
\[0 \to \uRad(M, N)  \to \Hom_{\uDiff(kQ/J^2\-\Proj)}(M, N) \to \Hom_{kQ^\op}(F(M), F(N))\to 0.\]
Since $k$ is a field, the previous exact sequence is split.
Then the desired isomorphism follows from Lemma ~\ref{lem:rad-ext}.
\end{proof}

By Theorem ~\ref{thm:repiso} we have the following; compare \cite[Corollary 4.10]{Wei2015}.

\begin{cor} \label{cor:bije}
The  functor $F$ gives a bijection from the isoclasses of objects in the homotopy category of
differential projective $kQ/J^2$-modules
to the isoclasses of objects in the category of $kQ^\op$-modules,
which carries indecomposable objects to indecomposable objects.
\end{cor}

\section{Compact generator}
Let $k$ be a field and $Q$ be a finite quiver.
Recall the opposite quiver $Q^\op$ and
the radical square algebra $kQ/J^2$ of $Q$.

Let $C$ be the $kQ/J^2$-module $kQ/J^2 \otimes_{kQ_0} kQ^\op$
with a differential $d$ given by
\[ d (y\otimes z)=\sum_{\alpha\in Q_1} y\alpha\otimes \alpha^*z\]
for $y\in  kQ/J^2$ and  $z\in kQ^\op$.
Here, we recall that $\alpha^*$ is the reversed arrow of $\alpha$.

Observe that $C$ is a reduced differential projective $kQ/J^2$-module.
Recall the functor $F$ from the previous section.
It is routine  to show that
$F(C)$ is isomorphic to the regular module over $kQ^\op$.

Recall the homotopy category $\uDiff(kQ/J^2\-\Proj)$ of differential projective $kQ/J^2$-module.
We will later show $C$ is a compact generator for this triangulated category.

Let $M$ be a differential projective $kQ/J^2$-module.
By Lemma ~\ref{lem:min}(2) there exists a decomposition $M = M'\coprod M''$
such that $M'$ is contractible and $M''$ is reduced.
Recall that the {\em cohomology group} $H(M)$ of $M$ is $\Ker d/\Im d$,
where $d$ is differential of $M$.
We also recall that $M$ is said to be exact if its cohomology group is zero.
Note that $H(M')=0$ and $H(M)=H(M'')$.

\begin{lem} \label{lem:fiso}
Let $M$ be a differential projective $kQ/J^2$-module. Then we have
\begin{enumerate}
  \item ${\uDiff(kQ/J^2\-\Proj)}(C,M) \simeq F(M'')$;
  \item $H(M)  \simeq  \Hom_{kQ^\op}(kQ_0, F(M'')) \coprod \Ext^1_{kQ^\op}(kQ_0, F(M''))$.
\end{enumerate}
\end{lem}

\begin{proof}
Note that $H(M)$ is isomorphic to $\uDiff(kQ/J^2\-\Proj)(kQ/J^2, M)$.
Here we denote by $kQ/J^2$ the differential module $kQ/J^2$ with vanishing differential.
Then (1) and (2)  follow from Theorem ~\ref{thm:repiso}(4).
\end{proof}

The ``Koszul dual functor" $F$ has a good restriction on some full subcategories.
More precisely, we have the following result.

\begin{prop} \label{prop:restric}
Let $M$ be a reduced differential projective $kQ/J^2$-module in the stable category $\uDiff(kQ/J^2\-\Proj)$.
\begin{enumerate}
  \item $M$ is finite dimensional if and only if $F(M)$ is  finite dimensional.
  \item $M$ is compact  if and only if $F(M)$ is finitely presented.
  \item $M$ is exact if and only if $\Ext^n_{kQ^\op}(kQ_0, F(M))=0$ for $n=0,1$.
\end{enumerate}
\end{prop}

\begin{proof}
(1) Since $F(M)$ is the top of the projective module $M$,
it follows that $M$ is finitely generated if and only if $F(M)$ is finitely generated.
Since  finitely generated $kQ/J^2$-modules are exactly  finite-dimensional  $kQ/J^2$-modules,
we infer that $M$ is finite dimensional if and only if $F(M)$ is  finite dimensional.

(2) $``\Longrightarrow"$
Assume that $M$ is compact.
Let $\{Y_\lambda\}_{\lambda\in L}$ be a set of $kQ^\op$-modules.
Since the functor $F$ is dense  by Theorem ~\ref{thm:repiso}(1),
every $Y_\lambda$ is isomorphic to $F(T_\lambda)$ for some
reduced differential projective $kQ/J^2$-module $T_\lambda$.

By Theorem ~\ref{thm:repiso}(2) and (4) there are isomorphisms
\[\Hom_{kQ^\op}(F(M), \coprod_{\lambda\in L} Y_\lambda) \simeq
\coprod_{\lambda\in L} \Hom_{kQ^\op}(F(M), Y_\lambda),\]
and
\[\Ext^1_{kQ^\op}(F(M), \coprod_{\lambda\in L} \presuper\sigma Y_\lambda) \simeq
\coprod_{\lambda\in L} \Ext^1_{kQ^\op}(F(M), \presuper\sigma Y_\lambda).\]
It follows from Lemma ~\ref{lem:comhe} that $F(M)$ is  finitely presented.

$``\Longleftarrow"$
Assume that $X=F(M)$ is a finitely presented $kQ^\op$-module.
Let us take a set $\{T_\lambda \}_{\lambda\in L}$  of  differential projective $kQ/J^2$-modules.
By Lemma ~\ref{lem:min}(2) we have $T_\lambda = T'_\lambda\coprod T''_\lambda$
such that $T_\lambda'$ is contractible and $T_\lambda''$ is reduced.

By Lemma ~\ref{lem:comhe} and Theorem ~\ref{thm:repiso}(4) we have isomorphisms
\[  {\uDiff(kQ/J^2\-\Proj)}(M, \coprod_{\lambda\in L} T_\lambda)
   \simeq  \coprod_{\lambda\in L}  {\uDiff(kQ/J^2\-\Proj)}(M,T_\lambda).\]
Then $M$ is compact in $\uDiff(kQ/J^2\-\Proj)$.

(3) This follows directly from Theorem ~\ref{thm:repiso}(4) and Lemma ~\ref{lem:fiso}(2).
\end{proof}

We have two full subcategories of  $\T=\uDiff(kQ/J^2\-\Proj)$.
Denote by $\T^{c}$ is the full subcategory formed by compact objects and
by $\T^{fd}$ is the  full subcategory
formed by objects  $M$ such that its reduced part $M''$ is finite dimensional.

\begin{cor}
We have an inclusion $\T^{fd}\subseteq \T^{c}$.
Moreover, the equality holds if and only if the quiver $Q$ is  acyclic.
\end{cor}

\begin{proof}
Recall that the category of all  finite-dimensional $kQ^\op$-modules is contained in
the category of all finitely presented  $kQ^\op$-modules, these two categories
coincide if and only if  the quiver $Q$ is acyclic.
Then the corollary follow directly from Proposition ~\ref{prop:restric}.
\end{proof}

By Theorem ~\ref{thm:repiso}
we have the following; compare ~\cite[Theorem 2]{RZ2017}.

\begin{cor}
The bijection in Corollary ~\ref{cor:bije} carries finite-dimensional objects to finite-dimensional objects and
carries compact objects to finitely presented objects.
\end{cor}

In particular, if we take $Q$ to be the $n$-loop quiver with $n\geq 2$,
then the bijection between finite-dimensional objects has already studied in \cite[Theorem 3.6]{Ran2015}.

Let $\T$ be a triangulated category admitting all small coproducts.
An object  $S$ in $\T$ is said to be  {\em compact}  ~\cite{Nee1996}
if for any set $\{T_\lambda\}_{\lambda \in L}$ of objects in $\T$, the natural monomorphism
\[\coprod_{\lambda \in L} \Hom_\T(S,  T_\lambda) \lto  \Hom_\T(S, \coprod_{\lambda \in L} T_\lambda)\]
is an epimorphism (and thus isomorphism).

Recall that a triangulated category $\T$ is said to be {\em compactly generated} ~\cite{Nee1996}
if $\T$ admits all small coproducts, and there exists a set $\S$ of objects in $\T$ such that
\begin{enumerate}
  \item Given  $T \in \T$, if $\T(\Sigma^n S,T) = 0$ for every $S \in \S$ and $n\in \bbZ$,
      then $T\simeq 0$;
  \item Every object $S \in\S$ is compact.
\end{enumerate}
Here, $\Sigma$ denotes the translation functor of $\T$.
The set $\S$ is called a {\em compact generating set} for $\T$.
In particular, if $\S = \{S_0\}$ is a singleton,
then  $S_0$ is called a {\em compact generator} for $\T$.

We now prove that $C$ is a compact generator for $\uDiff(kQ/J^2\-\Proj)$.

\begin{thm} \label{thm:comgen}
The homotopy category $\uDiff(kQ/J^2\-\Proj)$ is compactly generated,
where  $C$ is a compact generator for it.
\end{thm}

\begin{proof}
Let $T$ be an object in $\uDiff(kQ/J^2\-\Proj)$.
We have  $T = T'\coprod T''$ such that $T'$ is contractible
and $T''$ is reduced by Lemma ~\ref{lem:min}(2).

Suppose ${\uDiff(kQ/J^2\-\Proj)}(C,T) = 0$.
Then $F(T'')= 0$ by Lemma ~\ref{lem:fiso}(1).
Since $T''$ is projective over $kQ/J^2$,
we have $T'' =0$ and $T  = T'$ is contractible.
Then $T \simeq 0$ in   $\uDiff(kQ/J^2\-\Proj)$.
We infer that $C$ is a generator for  $\uDiff(kQ/J^2\-\Proj)$.

Take a set  $\{T_\lambda\}_{\lambda\in L}$ of objects in $\uDiff(kQ/J^2\-\Proj)$.
Then  Lemma ~\ref{lem:min}(2) yields that $T_\lambda = T'_\lambda\coprod T''_\lambda$,
where $T_\lambda'$ is contractible and $T_\lambda''$ is reduced.

By  Lemma ~\ref{lem:fiso}(1) there are isomorphisms
\[
{\uDiff(kQ/J^2\-\Proj)}(C, \coprod_{\lambda\in L} T_\lambda)
 \simeq  F(\coprod_{\lambda\in L} T''_\lambda)     \]
and
\[  \coprod_{\lambda\in L} {\uDiff(kQ/J^2\-\Proj)}(C,T_\lambda)
\simeq \coprod_{\lambda\in L} F(T''_\lambda).\]
Recall from Theorem ~\ref{thm:repiso}(2) that $F$ commutes with all small coproducts.
It follows that $C$ is a compact object in $\uDiff(kQ/J^2\-\Proj)$.

Therefore  $C$ is a compact generator for $\uDiff(kQ/J^2\-\Proj)$.
\end{proof}

\section{Virtually Gorenstein algebras}
In this section, we study the virtue Gorensteiness of algebras.
Here, we recall the notion of
Gorenstein projective modules and Gorenstein injective modules;
see ~\cite{EJ1995} for more details.

Given a ring $\Lambda$,
a complex $P^\bullet$ of  projective $\Lambda$-modules
is said to be {\em totally acyclic}  if $P^\bullet$ is acyclic and
$\Hom_\Lambda(P^\bullet, T)$ is acyclic for every projective $\Lambda$-module $T$.
A $\Lambda$-module $M$ is said to be {\em Gorenstein projective}   if
there exists a totally acyclic complex $P^\bullet$ of  projective $\Lambda$-modules
such that $M$ is isomorphic to $\Coker (d^{-1}\colon P^{-1} \to P^{0})$.
The complex $P^\bullet$ is called a {\em complete projective resolution} of $M$.

Dually, a complex $I^\bullet$ of  injective $\Lambda$-modules
is said to be {\em totally acyclic}  if $I^\bullet$ is acyclic and
$\Hom_\Lambda(T,I^\bullet)$ is acyclic for every injective $\Lambda$-module $T$.
A $\Lambda$-module $M$ is said to be {\em Gorenstein injective}   if
there exists a totally acyclic complex $I^\bullet$ of  injective $\Lambda$-modules
such that $M$ is isomorphic to $\Ker (d^{0}\colon I^{0} \to I^{1})$.
The complex $I^\bullet$ is called a {\em complete injective resolution} of $M$.

Denote by $\Lambda\-\Mod$ the category of all $\Lambda$-modules.
Let us  denote by  $\Lambda\-\Proj$ the full subcategory of projective $\Lambda$-modules and
denote  by $\Lambda\-\GProj$ the full subcategory of  Gorenstein projective $\Lambda$-modules.

We need the following facts; see ~\cite{Buc1987,EJ2000} for more details.

\begin{enumerate}
  \item $\Lambda\-\GProj$ is a Frobenius category with $\proj(\Lambda\-\GProj)=\Lambda\-\Proj$.
  \item The stable category $\Lambda\-\uGProj$ is a triangulated category.
  \item  $\Lambda$ is a quasi-Frobenius ring if and only if $\Lambda\-\GProj = \Lambda\-\Mod$.
  \item If the left global dimension of $\Lambda$ is finite, then $\Lambda\-\GProj = \Lambda\-\Proj$.
\end{enumerate}

Let $\Lambda[\epsilon] = \Lambda[T]/\langle T^2\rangle$ be the ring  of dual numbers over $\Lambda$.
Note that differential modules over $\Lambda$ are just  modules over $\Lambda[\epsilon]$.
In particular, if $\Lambda$ is an algebra over a field $k$,
then the  algebra $\Lambda[\epsilon]$ is isomorphic to the tensor product
$ \Lambda \otimes_k k[\epsilon]$ of algebras.
Here, we recall that $k[\epsilon]=k[T]/\langle T^2\rangle$ is the algebra of dual numbers over $k$;
it is a selfinjective algebra.

By ~\cite[Theorem 1.1]{Wei2015} a differential $\Lambda$-module
$(M, d)$ is Gorenstein projective if and only if the underlying $\Lambda$-module $M$ is Gorenstein projective.
Then we know that every differential projective $\Lambda$-module is Gorenstein projective
since every projective module is Gorenstein projective.

Let $\Lambda$ be an Artin algebra.
Recall that  $\Lambda$ is said to  {\em Gorenstein} ~\cite{AR1991}
if the injective dimension  of $\Lambda$ and the injective dimension of $\Lambda^\op$ are both finite.
Algebras of finite global dimension  and selfinjective algebras are Gorenstein algebras.

We also recall that $\Lambda$ is said to be {\em virtually Gorenstein} ~\cite{Bel2005} if for every $\Lambda$-module $M$,
the functor $\Ext^1_\Lambda(-,M)$ vanishes on all Gorenstein projective $\Lambda$-modules if and only if
the functor $\Ext^1_\Lambda(M,-)$ vanishes on all Gorenstein injective $\Lambda$-modules.
Algebras of finite representation type and   Gorenstein algebras
are  virtually Gorenstein algebras.
Examples of non-virtually Gorenstein algebras can be founded in ~\cite{BK2008,Yos2003}.

We need the following.

\begin{lem} \label{lem:gorvir}
Let $\Lambda$ and $\Gamma$ be finite-dimensional algebras over a field $k$.
\begin{enumerate}
  \item $\Lambda$ is virtually Gorenstein if and only if every reduced compact object in the stable category
  $\Lambda\-\uGProj$ is finite dimensional.
  \item $\Lambda\otimes_k\Gamma$ is   Gorenstein  if and only if $\Lambda$ and $\Gamma$ are  Gorenstein.
   \item $\Lambda\otimes_k\Gamma$ is  selfinjective  if and only if $\Lambda$ and $\Gamma$ are selfinjective.
\end{enumerate}
\end{lem}

\begin{proof}
(1) This follows ~\cite[Theorem 8.2]{Bel2005}; see also ~\cite[Theorem 4]{BK2008}.

(2) and (3) These are taken from ~\cite[Proposition 2.2]{AR1991}.
\end{proof}

Let $k$ be a field and $Q$ be a finite quiver.
We investigate the Gorenstein projective differential modules over the radical square zero algebra $kQ/J^2$ of $Q$.
Recall that a finite connected quiver $Q$ is a basic cycle if the number of vertices
is equal to the number of arrows in $Q$
and all arrows  form an oriented cycle.

\begin{prop} \label{prop:diffgor}
Let $k$ be a field and  $Q$ be a finite connected quiver.
\begin{enumerate}
  \item If $Q$ is not a basic cycle, then  Gorenstein projective differential $kQ/J^2$-modules
        are just differential projective $kQ/J^2$-modules.
  \item Otherwise, every differential $kQ/J^2$-module is Gorenstein projective.
\end{enumerate}
\end{prop}

\begin{proof}
(1) Since $Q$ is not a basic cycle, by ~\cite[Theorem 2]{RX2012}
all Gorenstein projective modules over $kQ/J^2$ are projective.
By ~\cite[Theorem 1.1]{Wei2015}  Gorenstein projective differential $kQ/J^2$-modules
are just  differential projective $kQ/J^2$-modules.

(2) Since $Q$ is a basic cycle, the algebra $kQ/J^2$ is selfinjective.
Then $kQ/J^2[\epsilon]$ is selfinjective and thus every differential $kQ/J^2$-module is  Gorenstein projective.
\end{proof}

The following theorem provides a class of noncommutative Artin algebras
that are not virtually Gorenstein;
compare ~\cite[Theorem 6.1]{Yos2003}.

\begin{thm} \label{thm:virgor}
Let $k$ be a field and  $Q$ be a finite connected quiver.
\begin{enumerate}
  \item If $Q$ is  acyclic, then the algebra $kQ/J^2[\epsilon]$ is Gorenstein.
  \item If $Q$ is a basic cycle, then the algebra $kQ/J^2[\epsilon]$ is selfinjective.
  \item Otherwise, the algebra $kQ/J^2[\epsilon]$ is not virtually Gorenstein.
\end{enumerate}
\end{thm}

\begin{proof}
(1) If $Q$ is  acyclic, then the algebra $kQ/J^2$  has finite global dimension.
By Lemma ~\ref{lem:gorvir}(2) the algebra $kQ/J^2[\epsilon]$ is Gorenstein.

(2) If  $Q$ is a basic cycle, then the algebra $kQ/J^2$ is selfinjective.
By Lemma ~\ref{lem:gorvir}(3) the algebra $kQ/J^2[\epsilon]$ is selfinjective.

(3) Since $Q$  is not a basic cycle,
by Proposition ~\ref{prop:diffgor}(1) the homotopy category of differential projective $kQ/J^2$-modules
is exactly the stable category  of Gorenstein projective $kQ/J^2[\epsilon]$-modules.

Since $Q$ has oriented cycles,
the compact generator in Theorem ~\ref{thm:comgen} is reduced and not finite dimensional.
It follows from Lemma ~\ref{lem:gorvir}(1) that the algebra $kQ/J^2[\epsilon]$ is not virtually Gorenstein.
\end{proof}

\begin{rmk}
Following ~\cite{RX2012} we know that the radical square zero algebra $kQ/J^2$
is virtually Gorenstein for every finite quiver $Q$.

Let $k$ be a field of characteristic 2, then
the algebra $kQ/J^2[\epsilon]$ is isomorphic to the
group algebra $(kQ/J^2)C_2$ where $C_2$ is the cyclic group of order 2.
Now if $Q$ is a quiver in Theorem ~\ref{thm:virgor}(3),
then the group algebra $(kQ/J^2)C_2$ is not virtually Gorenstein; compare \cite[Proposition 3.1]{BS2015}.
\end{rmk}

We end this section by an example.

\begin{exm}
Let $k$ be a field and $Q, Q'$ be the following quivers.
\begin{align*}
Q\colon\quad \xymatrix{{}^\bullet_1 \ar@(ur,ul)_\alpha \ar[r]^{\beta}  &
{}^\bullet_2}& \qquad \qquad Q'\colon\quad \xymatrix{{}^\bullet_1  \ar@(ur,ul)_\alpha\ar@(ul,dl)_{\epsilon_1} \ar[r]^{\beta}
&{}^\bullet_2 \ar@(ur,dr)^{\epsilon_2}}
 \end{align*}

Let $I$ be the ideal of $kQ'$ generated by
$\{\alpha^2,\beta\alpha,\epsilon_1^2,\epsilon_2^2,\alpha\epsilon_1-\epsilon_1\alpha, \beta\epsilon_1-\epsilon_2\beta\}$.
Then the   algebras $kQ'/I$  and $kQ/J^2[\epsilon]$ are isomorphic.

By Theorem ~\ref{thm:virgor}(2) the algebra $kQ'/I$ is not virtually Gorenstein.
\end{exm}

{\bf\noindent Acknowledgements.}
The author is very grateful  to Prof. ~Xiao-Wu ~Chen and Prof. ~Yu ~Ye for numerous inspiring ideas.
The author thanks Dr. ~Zhe ~Han and Dr. Bo ~Hou for many discussions in the seminar at Henan University.
The work is supported by the  Natural Science Foundation of China  (No. 11571329).

\end{document}